\documentclass[12pt]{article}

\usepackage{graphicx}

\usepackage{amsmath,amsthm,amssymb,amsfonts,color}
\usepackage[normalem]{ulem}
\usepackage{tikz}
\usetikzlibrary{positioning}
\usepackage[font=small]{caption}

\makeatletter
\g@addto@macro\normalsize{%
  \setlength\abovedisplayskip{8pt plus 3pt minus 3pt}
  \setlength\belowdisplayskip{8pt plus 3pt minus 3pt}
  \setlength\abovedisplayshortskip{6pt plus 3pt minus 2pt}
  \setlength\belowdisplayshortskip{6pt plus 3pt minus 2pt}
}
\makeatother

\date{\today}

\def\leq{\leqslant}
\def\geq{\geqslant}
\def\le{\leqslant}

\normalsize

\numberwithin{equation}{section}

\def\({\bigl(}
\def\){\bigr)}

\newtheorem{thm}{Theorem}
\numberwithin{thm}{section}
\newtheorem{cor}[thm]{Corollary}

\newtheorem{lemma}[thm]{Lemma}

\theoremstyle{definition}
\newtheorem{remark}[thm]{Remark}

\newcommand{\A}{\mathbf A}

\def\dfrac#1#2{\lower0.15ex\hbox{\large$\textstyle\frac{#1}{#2}$}}
\def\({\bigl(}
\def\){\bigr)}
\def\st{\mathrel{:}}

\let\eps=\varepsilon

\newcommand{\D}{\mathbf D}
\def\calE{\mathcal{E}}

\def\X{\boldsymbol{X}}

\def\calN{\mathcal{N}}

\def\calG{\mathcal{G}}

\def\dvec{\boldsymbol{d}}

\def\({\bigl(}
\def\){\bigr)}

\def\E{\operatorname{\mathbb{E}}}

\def\Reals{{\mathbb{R}}}

\def\Naturals{{\mathbb{N}}}

\def\Bin{\operatorname{Bin}}

\def\GoodS{S_{\rm good}}
\def\GoodT{T_{\rm good}}

\title{On the maximum number of common neighbours\\ in dense random regular graphs}

\author{
Mikhail Isaev\thanks{Research supported by  ARC DP220103074} \\
\small 
 School of Mathematics 
and Statistics
 \\[-0.8ex]
\small  UNSW Sydney\\[-0.8ex]
\small  Sydney, NSW, Australia\\[-0.8ex]
\small\tt  m.isaev@unsw.edu.au 
\and
  Maksim Zhukovskii\\
  \small School of Computer Science\\[-0.8ex]
\small University of Sheffield\\[-0.8ex]
\small  Sheffield, UK\\[-0.8ex]
\small\tt   m.zhukovskii@sheffield.ac.uk
}
\date{}

\begin{document}

\maketitle

\abstract{
We derive the distribution of  the maximum  number of common neighbours of a pair of vertices in a dense random regular graph.
The proof involves  two important steps.
 One step is to establish the extremal independence property: the asymptotic equivalence  with the maximum  component of a vector  with independent marginal distributions.
 The other step is to prove that the distribution of the number of common neighbours for each pair of vertices can be approximated by the binomial distribution.
}

\section{Introduction}

The distribution of the degree sequence in a dense binomial random graph $\calG(n,p)$ (that is, for constant $p\in (0,1)$) was thoroughly studied by Bollob\'{a}s in \cite{Bol1,Bol2}. In particular, the maximum degree after an appropriate scaling converges in distribution to 
the standard Gumbel distribution, which is an absolutely continuous distribution on $\mathbb{R}$ with the cdf $e^{-e^{-x}}$.
Even earlier, Ivchenko \cite{Ivchenko1973} showed that 
this is also true for the sparse case. 
An explanation of this limiting behaviour is that   the degree sequence of $\calG(n,p)$ is close to a vector of independent binomial random variables (see  McKay, Wormald~\cite{MW-indep} for a more detailed discussion), while the distribution of the maximum component of  such vectors   is well-studied in the extreme value theory (see, for example,  Nadarajah, Mitov~\cite{nadarajah2002}). Various extensions to other random models and more general extremal graph characteristics were considered in \cite{AGSW,CZL2012, IRZZ, Palka1987, RZ, Shang2012}.

In this paper, we consider the random graph $\calG_{n,d}$ chosen uniformly 
at random from the set of vertex-labelled $d$-regular graphs on $[n]:=\{1,\ldots,n\}$. Since all degrees are the same in this random graph model,  it is natural to consider the distribution of the maximum (or minumum) $2$-degree, which is the  number of common neighbours of a pair of vertices.
Our results show  that its limit distribution is also Gumbel, at least  in the dense case.

It is worth noting that limiting distributions for statistics  in dense random regular graphs are typically much harder to determine than in binomial random graphs due to dependencies for adjacencies and pattern appearances. 
Actually, we are aware of only one result of such kind, namely, the asymptotic normality of the number of subgraphs isomorphic to a given graph of a constant size due to Sah and Sawney~\cite{SS}, and for subgraphs of growing size nothing is known, see~\cite{Gao} for further discussion.

We proceed to the formal statement of our result.  
Let $d=d(n)$ be a sequence of non-negative integers
such that      
\begin{equation}\label{ass-main}
       d   = \lambda (n-1) \in \Naturals, \qquad  \text{$dn$ is even,}   \qquad \lambda(1-\lambda) = \Omega(1).
\end{equation}
For a graph $G$ on vertex set $[n]$ and a vertex $i \in [n]$, let $N_{i}(G) \subset [n]$ denote the  set of neighbours of  $i$ in  $G$. 
For $ij \in \binom{[n]}{2}$, let 
\[
    X_{ij}(n,d):= \left|N_i(\calG_{n,d}) \cap N_j(\calG_{n,d})\right|.
\]
 The following result establishes  the joint limiting distribution   of 
\begin{equation}\label{def:max-min}
    X_{\max}(n,d) =  \max_{ij \in \binom{[n]}{2}}  X_{ij}(n,d), \qquad 
   X_{\min}(n,d) =  \min_{ij \in \binom{[n]}{2}}  X_{ij}(n,d).   
\end{equation}


\begin{thm}\label{T:main}
Let  \eqref{ass-main} hold and
  $a_{n,d}, b_{n,d}$ be defined by
\begin{equation}
\begin{aligned}
 		a_{n,d} &:= \lambda^2 n + 2\lambda (1-\lambda) \sqrt{ n   \log n } \left(1- \dfrac{\log \log n}{8 \log n} - 
 		\dfrac{\log (32\pi)}{8 \log n}\right),
 		\\ b_{n,d} &:=  \dfrac12 \lambda(1-\lambda) \sqrt{\dfrac{n}{\log n}}.
\label{eq:a_n-b_n}
\end{aligned}
\end{equation}
Then,  as $n\to\infty$, the vector $\left(\dfrac{X_{\max}(n,d) -a_{n,d}}{b_{n,d}},\dfrac{2\lambda^2n-a_{n,d}-X_{\min}(n,d) }{b_{n,d}}\right)$ converges in distribution to a pair of  independent standard Gumbel random variables.

\end{thm}

We prove Theorem \ref{T:main} in  Section \ref{S:mainBig} using a general framework for  limiting distribution of the maximum  of dependent random variables  recently suggested by Isaev, Rodionov, Zhang, Zhukovskii   \cite{IRZZ}. This framework assumes that most variables are weakly dependent and satisfy a certain $\varphi$-mixing condition, while allowing strong dependencies between a few variables represented as some kind of  \emph{dependency graph}; see Section \ref{S:extremal} for the precise statement. A big advantage of the framework of  \cite{IRZZ} in application to extremal characteristics of random regular graphs is to avoid computation of high order moments, which can be very hard or even unfeasible.  This paper     adopts this framework to our particular problem, but
we   believe that 
the arguments of present work extend to many other similar problems for random regular graphs and other random discrete structures with intricate dependencies.

Finally, we recall that Babai, Erd\H{o}s,   Selkow~\cite{BES}  used the degree distribution of the binomial random graph   to prove the existence of linear-time algorithm that canonically labels almost all graphs, which implies that  graph isomorphism can be tested in linear time for almost all graphs.
Thus, the study of maximal number of common neighbours in random regular graphs   is also interesting from the algorithmic perspective, since the pair of vertices achieving maximal number of common neighbours can be used for   labelling and isomorphism testing.  To our knowledge, the existence of an efficient canonical labelling algorithm for asymptotically almost all dense regular graphs is an open question.  

 In the next section, we present some additional results that will appear in the proof of Theorem~\ref{T:main}, which can be of independent interest.


\subsection{Useful ingredients appearing in the proof of Theorem~\ref{T:main}}\label{S:ingredients}

Using enumeration results of McKay \cite{McKay2011}, 
we   derive a  local limit theorem for the number of common neighbours of a particular pair of vertices in $\calG_{n,d}$.  
Even though similar computations appeared in the literature before,  see, for example,  Krivelevich, Sudakov, Vu,  Wormald in~\cite{KSVW2001}, but we could not find a result strong enough for our purposes.

 \begin{thm}\label{T:distribution_pair}
 Let  \eqref{ass-main} hold. Then, for any distinct  $ij \in \binom{[n]}{2}$,   uniformly over $0\leq h\leq d$,
 \begin{itemize}
  \item[(a)]
 	$\displaystyle
 	\Pr\Big(X_{ij}(n,d) = h \mid ij \notin \calG_{n,d}\Big)  \sim  \frac{ \binom{d}{h}  \binom{n-d-2}{d-h} } {\binom{n-2}{d}}   \exp\left(\dfrac{\lambda}{1-\lambda} -   \dfrac{h} {\lambda(1-\lambda)n}\right);
 $
 \item[(b)]
  	$\displaystyle
 	\Pr\Big(X_{ij}(n,d) = h \mid ij \in \calG_{n,d}\Big)  \sim  \frac{ \binom{d-1}{h}  \binom{n-d-1}{d-h-1} } {\binom{n-2}{d-1}}   \exp\left(\dfrac{\lambda}{1-\lambda} -   \dfrac{h} {\lambda(1-\lambda)n}\right),
 $

\item [(c)]  In addition, if $h \sim \lambda^2n$ then 
$\displaystyle
 	  		\Pr\Big(X_{ij}(n,d) = h\Big)  	
 	  		    \sim  	\frac{\binom{d}{h} \binom{n-1-d}{d-h}}{ \binom{n-1}{d} }.
 	  	$

\end{itemize}
\end{thm}

In fact, to prove Theorem \ref{T:main}, we only need the following corollary of Theorem \ref{T:distribution_pair}  that shows that the distribution of $X_{ij}(n,d)$ can be approximated by the binomial distribution.
 \begin{cor}\label{L:dist}
 Let $ d = \lambda (n-1) \in \Naturals$  be  such that $dn$ is even and  $\lambda(1-\lambda) = \Omega(1)$. 
 Then, for any distinct  $ij \in \binom{[n]}{2}$, the following  hold.
 	\begin{itemize}
 	\item[(a)] 
 	 For all integer $h$ such that  
 	  $h =    \left(1+ O\left(\dfrac{\log n }{\sqrt{n}}\right)\right)\lambda^2 n$, 
 	  \[
 	  	\Pr \Big( X_{ij}(n,d) = h\Big)    \sim  \Pr(\xi  = h),
 	  \]
 	  where $\xi$ is  distributed according to $\Bin(N,p)$ with
 	  \[
 	  			N := \left\lfloor \dfrac{\lambda}{2-\lambda}n \right\rfloor, \qquad p:=\lambda(2-\lambda). 
 	  \]  
 \item[(b)]
    With probability at least $1 - e^{-\omega(\log n)}$, 
 	\[
 	\Big| X_{ij}(n,d) -\lambda^2 n \Big|   \leq \sqrt{n} \log n.
 	 	\]
  	  \end{itemize}	
 \end{cor}

We prove Theorem~\ref{T:distribution_pair} and Corollary \ref{L:dist} in Section~\ref{S:distribution}.  In particular, our proof relies on the next result, which is useful  to relate $\calG_{n,d}$  and its conditional version  given the number of common neighbours of any two vertices without significant change in the graph structure. We need it to
 establish the aforementioned  $\varphi$-mixing condition for the framework of \cite{IRZZ}.
\begin{thm}\label{t:phi}
  		   Let the assumptions of Theorem \ref{T:main} hold and let $h\in [d]$ satisfy
    \[  
  	h  = \left (1+ O\left(\frac{\log n}{\sqrt{n}}\right) \right)\lambda^2 n \]
   Let 
    $i,j\in [n]$ be any two vertices of 
  $\calG_{n,d}$.
  		Then, there exists a coupling  $(\calG_{n,d},\calG_{n,d}^h)$ satisfying the following two conditions:
  		\begin{itemize}
  				\item[(i)] $\calG_{n,d}^h$ is uniformly distributed on the set of $d$-regular graphs 
  				such that $i$ and $j$ have exactly $h$ common neighbours;
  	 \item[(ii)] with probability $1 -o\left(\dfrac{\log^2 n}{\sqrt{n}}\right)$,
  		the neighbourhoods of each  vertex $k \notin [n]\setminus\{i,j\}$ in graphs $\calG_{n,d}$ and  $\calG_{n,d}^h$ differ by at most 8 elements.
  		\end{itemize}
  	\end{thm}

We prove Theorem \ref{t:phi} in Section~\ref{S:cond}.


\subsection{Structure of the paper}
The paper  is organised as follows.
 In Section~\ref{sc:pre}, we recall the result 
 of Isaev, Rodionov, Zhang, Zhukovskii   \cite{IRZZ} that gives sufficient conditions for the extremal independence property: the distribution of the maximum of dependent random variables is asymptotically equivalent to   the distribution of the maximum of their independent copies.
 We also derive some probability estimates for almost regular random graphs that will appear repeatedly in the proofs and  are somewhat straightforward from known results  available in the literature.

In Section~\ref{S:mainBig}, we prove Theorem \ref{T:main}. Furthermore, we prove convergence rates of the order $o\left(\frac{\log^2 n}{\sqrt{n}}\right)$  for the extremal independence property for the numbers of common neighbours in $\calG_{n,d}$.  The arguments in  Section~\ref{S:mainBig} rely on the results presented in Section \ref{S:ingredients}, whose proofs are given in further sections.

In   Section~\ref{S:distribution},   we   prove our local limit results for the number of common neighbours of a pair of vertices,  Theorem \ref{T:distribution_pair} 
 and  Corollary \ref{L:dist}.  The proof of Theorem  \ref{T:distribution_pair} is by exposing the neighbourhood of a vertex and then applying  the  estimates for almost regular random graphs   from Section \ref{S:almost}.

In Section~\ref{S:cond}, we prove Theorem  \ref{t:phi}. Generalising \cite[Theorem~2.1]{IMSZ},  we give an abstract result  on the existence of coupling in a bipartite graph most of whose vertices have degrees that are not too small in comparison to the average.
 Using  this abstract result, we get
  a coupling $(\calG_{n,d}^h, \calG_{n,d}^{h+1})$  that does not change much the graph structure. 
Then, we combine several couplings $(\calG_{n,d}^h, \calG_{n,d}^{h+1})$ to get the desired coupling  of  $\calG_{n,d}$ and  $\calG_{n,d}^h$.


\section{Preliminaries}
\label{sc:pre}

Here, we collect all preliminary results that we use in the proofs. 
In Section \ref{S:extremal}, we state  sufficient conditions from \cite{IRZZ}  for a random vector $\X  = (X_1,\ldots,X_m)^T$ to satisfy the extremal independence property: as $n \rightarrow \infty$
\begin{equation}
\left| \Pr\left( \max_{i \in [m]} X_i \leq x \right)
- \prod_{i\in [m]} \Pr \left( X_i \leq x \right) \right|
\longrightarrow 0 \qquad \text{for any fixed $x\in\Reals$,}
\label{a_e_indep_formula}
\end{equation}
where $m=m(n) \in \Naturals$ and $X_i= X_i(n)\in \Reals$ for all $i \in [m]$.
In Section \ref{S:almost}, we give  probability bounds for random almost regular graphs  needed to verify these sufficient conditions in application to the maximum number of common neighbours. 

All asymptotics in this paper refer to the passage of $n$ to infinity and the notations
$o(\cdot)$, $O(\cdot)$, $\Omega(\cdot)$ have the standard meaning.
We also use the notation $f(n)\sim g(n)$ when $f(n) = (1+o(1))g(n)$.

\subsection{Estimates from the extreme value theory}\label{S:extremal}

The extremal independence property (\ref{a_e_indep_formula}) is equivalent to
\begin{equation}\label{a_e_indep_formula_for_A}
\left|\Pr\left( \bigcap_{i \in [m]} \overline{ A_i } \right)-\prod_{i\in [m]} \Pr\left(\overline{ A_i } \right)\right|\longrightarrow 0,  
\end{equation}
where
the system of events $\A$ is defined by
\begin{align}
\A   := (A_i)_{i \in [m]},
\quad
A_i := \{ X_i > x \},
\label{event}
\end{align}
and $\overline{ A_i } $ is the complement event of $A_i$.  Throughout this section, we always assume that $\Pr(A_i)>0$ for all $i\in [m]$, as the presence of events of probability zero makes no difference for \eqref{a_e_indep_formula_for_A}.

We represent the dependencies among the events of $\A$ by a graph $\D$
on the vertex set $[m]$ with edges
indicating the pairs of `strongly dependent' events,
while non-adjacent vertices correspond to `weakly dependent' events.
Let  $D_i \subseteq [m]$ be the closed neighbourhood of vertex $i$ in graph $\D$.
We allow $\D$ to be a directed graph, that is,
there might exist $i, j \in [m]$, such that $i \in D_j$ and $j \not\in D_i$.
The quality of the representation of the dependencies for $\A$
by a graph $\D$ is measured by the following mixing coefficient:
\begin{align}\label{phi}
\varphi(\A, \D)
:= \max_{i \in [m]}  \left|
\Pr\left(\bigcup_{j \in [i-1] \setminus D_i   } A_j    \mid A_i\right)
- \Pr\left(\bigcup_{j \in [i-1] \setminus D_i   } A_j   \right)  \right|.
\end{align}
The influence of `strongly dependent' events is measured by declustering coefficients
$\Delta_1$ and $\Delta_2$ defined by
\begin{align}
\Delta_1(\A, \D)
&:= \sum_{i\in[m]}\sum_{j \in [i - 1] \cap D_i }\Pr( A_i\cap A_j),
\label{del1} \\
\Delta_2(\A, \D)
&:= \sum_{i\in[m]}\sum_{j \in [i - 1] \cap D_i }\Pr( A_i)\Pr(A_j).
\label{del2}
\end{align}
The choice of graph $\D$
leads to the trade-off between
the mixing coefficient $\varphi(\A, \D)$
and declustering coefficients $\Delta_1(\A, \D)$ and $\Delta_2(\A, \D)$,
since $\Delta_1(\A, \D)$ and $\Delta_2(\A, \D)$
increase as $\D$ gets denser, and $\varphi(\A, \D)$ typically decreases.

Our main tool for Theorem \ref{T:main} is the following bound, which is a simplified version of \cite[Theorem 2.1]{IRZZ}.

\begin{thm}[Isaev, Rodionov, Zhang, Zhukovskii \cite{IRZZ}]\label{T:extremal}
For any system $\A = (A_i)_{i \in [m]}$
and graph $\D$ with vertex set $[m]$,
the following bound holds:
\begin{align}
\left| \Pr\left( \bigcap_{i \in [m]} \overline{ A_i } \right)
- \prod_{i\in [m]} \Pr\left(\overline{ A_i } \right) \right|
\le \left(1 - \prod_{i\in [m]} \Pr\left(\overline{ A_i } \right) \right) \varphi +
\max\{ \Delta_1, \Delta_2 \},
\label{eq:Tmain}
\end{align}
where $\varphi = \varphi(\A, \D)$,
$\Delta_1 = \Delta_1(\A, \D)$,
and $\Delta_2 = \Delta_2(\A, \D)$.
\end{thm}

We
will also need 
\cite[Theorem 3]{nadarajah2002}, stated below for reader's convenience.
This theorem  establishes the
 distribution of the maximum of independent binomial random variables. 
\begin{thm}[Nadarajah, Mitov \cite{nadarajah2002}]\label{T:bin}Let $p=p(n)\in (0,1)$ and $p(1-p)=\Theta(1)$. Also let $N=N(n) \in \Naturals$ and  $m=m(n) \in \Naturals$  satisfy
\[
	     N \gg \log^3 m \gg 1.  
\]
If $\xi_1,\ldots,\xi_m$ are $\Bin(N,p)$ independent random variables
  then $\left(\max_{i\in[m]}\xi_i-a_n^*\right)/b_n^*$ converges in distribution to a standard Gumbel random variable with  $a_n^*$ and $b_n^*$ defined by
 \begin{equation}\label{def:ab}
 \begin{aligned}
 	 a_n^*&= a_n^*(N,m,p):= pN +\sqrt{2 Np(1-p)  \log m}\left(1 - \dfrac{\log\log m}{4\log m} - \dfrac{\log( 2 \sqrt \pi)}{2\log m}\right),
 \\
 	b_n^* &= b_n^*(N,m,p):= \sqrt{\frac{ Np(1-p) }{2 \log m} }.
 \end{aligned}
 \end{equation}
%
 \end{thm}
Note that the original version of this theorem in \cite{nadarajah2002} is stated for a fixed $p\in(0,1)$, though exactly the same proof works for $p=p(n)$ bounded away both from 0 and 1.

\subsection{Estimates for random almost regular graphs}\label{S:almost}
A degree sequence $\dvec =\dvec(n)  =  (d_1, \ldots, d_n)^T$ is \emph{almost $d$-regular}, where $d= d(n)$, if 
\begin{equation*}
\max_{i \in [n] } |d_i - d|  = O(1).
 \end{equation*}   
We always assume the following: 
\begin{equation}\label{deg:ass}
\begin{aligned}
	 &\dvec =\dvec(n) \in \Naturals^n \text{ is almost $d$-regular degree sequence and} \\   &\lambda  (1-\lambda ) = \Omega(1),  \text{ where  $\lambda := d/(n-1)$.}
 \end{aligned}
\end{equation}
Under these assumptions,    the number  of graphs with degree sequence $\dvec$, denoted by $\calN(n,\dvec)$,  is a straightforward application of   \cite[Theorem 3]{McKay2011}:
\begin{equation}\label{eq:enumeration}
\calN(n,\dvec) \sim \sqrt{2} e^{1/4} \left( \left(\dfrac{\bar d}{n-1}\right)^{\frac{\bar d}{n-1}}\left(1-\dfrac{\bar d}{n-1}\right)^{1-\frac{\bar d}{n-1}}\right)^{\binom{n}{2}} \prod_{j\in [n]} \binom{n-1}{d_j},
\end{equation}
where 
\begin{equation*} 
	\bar{d}= \bar d(\dvec) := \frac{d_1+\cdots +d_n}{n}.
\end{equation*}
 Let $\calG_{\dvec}$ denote a uniform random graph with degree sequence $\dvec$.
 Recall that, for a graph $G$ on vertex set $[n]$ and a vertex $i\in [n]$, 
 $N_{i}(G) \subset [n]$ denotes  the  set of neighbours of $i$ in $G$.  
\begin{lemma}\label{l:neighb}
 Under assumptions \eqref{deg:ass}, we have 
	\begin{align*}
			\Pr(N_i(\calG_{\dvec}) =A) &\sim
 	\sqrt{2\pi \lambda (1-\lambda) n}  \ 
 	  \prod_{j \in A} \dfrac{ d_j}{n-1}  \prod_{j \notin A \cup\{i\} } \left(1 -  \dfrac{d_j}{n-1}\right)\\
    &= \Theta\left(\sqrt{ n}\, \lambda^{|A|} (1-\lambda)^{n-|A|} \right),
	\end{align*}
	uniformly over all choices $i \in [n]$, $A\subset [n]\setminus\{i\}$  with $|A|=d_i$ and all choices of $\dvec$ (with only dependency on the implicit constant in \eqref{deg:ass}).
 
\end{lemma}

\begin{proof}
  Observe that 
  \[
      \Pr(N_i(\calG_{\dvec}) =A) = \frac{\calN(n-1,\dvec')}{\calN(n,\dvec)},
  \]
  where $\dvec'$ is obtained from $\dvec$ by removing $i$'th component and reducing all components corresponding  to $A$ by one.
  The proof is by applying formula \eqref{eq:enumeration} to both numerator and denominator. 
  To estimate the ratio, we observe 
\[
	\dfrac{\bar d'}{n-2} =  \dfrac{\bar d n - 2d_i}{(n-1)(n-2)}  
	=    \dfrac{\bar d}{n-1}   + \dfrac{  \bar d-d_i}{  \binom{n-1}{2}},
\]
where, with a slight abuse of notation, we let
$\bar d= \bar d(\dvec)$ and $\bar d'=\bar d(\dvec')$. This gives
 \begin{align*}
 	&\left(\left(\dfrac{\bar d'}{n-2}\right)^{\frac{\bar d'}{n-2}} 
  \left(1-\dfrac{\bar d'}{n-2}\right)^{1-\frac{\bar d'}{n-2}}\right)^{\binom{n-1}{2}}
 	\\ 
  &\hspace{2cm}\sim \left( \left(\dfrac{\bar d}{n-1}\right)^{\frac{\bar d}{n-1}}\left(1-\dfrac{\bar d}{n-1}\right)^{1-\frac{\bar d}{n-1}}\right)^{\binom{n-1}{2}} 
 	\left( \dfrac{\bar d}{n-1-\bar d}\right)^{\bar d -d_i}.
 \end{align*}
Note  that assumption \eqref{deg:ass}  implies  $\left( \frac{\bar d}{n-1-\bar d}\right)^{\bar d -d_i} \sim \left(\frac{\lambda}{1-\lambda}\right)^{\bar d -d_i}$.
Using Stirling's approximation, we  estimate
\begin{align*}
	\binom{n-1}{d_i}
 \sim \frac{ (\frac{n-1}{d_i})^{d_i} (\frac{n-1}{n-1-d_i})^{n-1-d_i}}{ \sqrt{2\pi \lambda(1-\lambda) n}} 
	\sim \frac{\lambda^{-d_i} (1-\lambda)^{-n+1 + d_i}}{ \sqrt{2\pi \lambda(1-\lambda) n}}.   
\end{align*}
Combining the asymptotic equivalencies established above and formula \eqref{eq:enumeration}, we get that 
\begin{align*}
 \frac{\calN(n-1,\dvec')}{\calN(n,\dvec)}
 &\sim \frac{\left( \left(\frac{\bar d'}{n-2}\right)^{\frac{\bar d'}{n-2}}\left(1-\frac{\bar d'}{n-2}\right)^{1-\frac{\bar d'}{n-2}}\right)^{\binom{n-1}{2}} \prod_{j\in [n]\setminus\{i\}} \binom{n-2}{d_j'}}
 {\left( \left(\frac{\bar d}{n-1}\right)^{\frac{\bar d}{n-1}}\left(1-\frac{\bar d}{n-1}\right)^{1-\frac{\bar d}{n-1}}\right)^{\binom{n}{2}} \prod_{j\in [n]} \binom{n-1}{d_j}}
 \\
 &\sim  \frac{ 
 \left( \frac{\lambda}{1-\bar\lambda}\right)^{\bar d -d_i}
 \binom{n-1}{d_i}^{-1}}
 {
 \left( \left(\frac{\bar d}{n-1}\right)^{\frac{\bar d}{n-1}}\left(1-\frac{\bar d}{n-1}\right)^{1-\frac{\bar d}{n-1}}\right)^{n-1}
 }\,  
		 \prod_{j \in A} \dfrac{ d_j}{n-1}  \prod_{j \notin A \cup\{i\} } \left(1 -  \dfrac{d_j}{n-1}\right)
 \\
 &\sim\sqrt{2\pi \lambda (1-\lambda) n}  \ 
 	  \prod_{j \in A} \dfrac{ d_j}{n-1}  \prod_{j \notin A \cup\{i\} } \left(1 -  \dfrac{d_j}{n-1}\right)
 \end{align*}
as claimed.
The second claim with  $\Theta(\cdot)$ follows from assumption \eqref{deg:ass}.
 \end{proof}

The following concentration result is a simple consequence  of  \cite[Theorem~5.15]{IM2018}.

\begin{lemma}\label{l:concentration}
   Let $Y$ be a set of vertex pairs  such that
	$
		|Y| =  \Omega(n^2). 
	$ Under assumptions~\eqref{deg:ass}, with probability at least $1- e^{-\omega(\log n)}$,
	 \[
	 	|Y \cap  \calG_{\dvec}| = \left(1+o\left( \dfrac{\log n}{n}\right)\right)  \mathbb{E}|Y \cap  \calG_{\dvec}|
	 	 \sim \lambda |Y|.
	 \]
\end{lemma}
\begin{proof}
From \cite[Theorem 5.15]{IM2018}, we have that, for some fixed $\check{c}>0$ and every $\gamma\geq 0$,
\[
   \Pr\left(\left||Y \cap  \calG_{\dvec}| -  \E \hat{X} \right|  \leq \gamma |Y|^{1/2} \right)
   \geq 1 - \check{c} \exp\left(-2\gamma\min\{\gamma, n^{1/6} (\log n)^{-3}\}\right),
\]
where $\hat{X}$ is a certain random variable, the origin of which  is not of importance for this argument.  
Since $|Y| = \Omega(n^2)$ taking $\gamma =  \log^{2/3}n$, we get that 
\[
|Y \cap  \calG_{\dvec}| -  \E \hat{X}   = o(n \log n)
 \]
with probability at least $1 - e^{-\omega(\log n)}$. Since $0<|Y \cap  \calG_{\dvec}|<n^2$ always, we get that 
\begin{equation}\label{eq:Y-calG}
\E |Y \cap  \calG_{\dvec}|-  \E \hat{X}   = o(n \log n).
\end{equation}

Note that to apply \cite[Theorem 5.15]{IM2018}, one also need to check that $\dvec$  is $\delta$-tame. To verify that, we use \cite[Theorem 2.1]{BH2013}  which states the following sufficient condition. If 
there are $0<\alpha<\beta<1$ such that 
$(\alpha+\beta)^2 < 4\alpha $ and 
\[
   \alpha(n-1) < d_i < \beta(n-1) \quad \text{ for all $i\in [n]$,}
\]
then $\dvec$ is $\delta$-tame for some $\delta = \delta(\alpha,\beta)>0$ provided $n>n_0(\alpha,\beta)$.
Furthermore, one can take
\begin{align*}
    n_0 &= \max\left\{\frac{\beta}{\alpha(1-\beta)}, \frac{4(\beta-\alpha)}{4\alpha - (\alpha+\beta)^2}\right\} +1,\\
    \delta &= \frac{\epsilon^6}{1+\epsilon^6}, \ \ 
    \text{where } \epsilon = \min\left\{\alpha, \alpha - \frac{(\alpha+\beta)^2}{4}\right\}.
\end{align*}
Under assumptions \eqref{deg:ass}, by  taking $\alpha = \lambda-\eps$ and $\beta=\lambda+\eps$ for sufficiently small $\eps$, we show that $\dvec$ is $\delta$-tame for $n \geq n_0$, where $n_0$ and $\delta$  depend only on the implicit constants in \eqref{deg:ass}. In particular,  $\lambda$ is not required to converge: we only need it is bounded away from 0 and 1. Thus, we get \eqref{eq:Y-calG}.

Finally, using \cite[Theorem 2.1]{McKay2011}, we find that all edge probabilities in $\calG_{\dvec}$ are asymptotically equivalent to $\lambda$, which implies 
\[
\E |Y \cap  \calG_{\dvec}| \sim \lambda |Y|.
\]
The claimed bounds follow.
\end{proof}





\section{Extremal independence for common neighbours}\label{S:mainBig}
 
In this section, we estimate the convergence rates for the extremal independence property for the vector of numbers of common neighbours  in random regular graph $\calG_{n,d}$ and, as a consequence, establish Theorem~\ref{T:main}.

Recall that $X_{ij}(n,d) = \left|N_i(\calG_{n,d}) \cap N_j(\calG_{n,d})\right|.$
We consider the joint distribution function $F:\Reals^2 \rightarrow [0,1]$  of the variables $X_{\max}(n,d)$  and  $-X_{\min}(n,d) $  from \eqref{def:max-min} after appropriate scaling: 
\[
        F_{n,d}(x,x'):= \Pr\left(X_{\max}(n,d) \leq a_{n,d} + b_{n,d}\, x \text{ and } X_{\min}(n,d) \geq 2\lambda^2 n -a_{n,d} + b_{n,d}\, x' \right),
\]
where $a_{n,d}$ and $b_{n,d}$ are defined in \eqref{eq:a_n-b_n}.
We show that   $F_{n,d}(x,x')$ can be approximated by the corresponding distribution function $\hat{F}: \Reals^2 \rightarrow [0,1]$ for independent copies of $X_{ij}(n,d)$, which can be defined by
 \begin{align*}
    \hat{F}_{n,d}(x,x') := \prod_{ij \in \binom{[n]}{2}} \Pr\left( 2\lambda^2 n -a_{n,d} + b_{n,d}\, x'  \leq X_{ij}(n,d) \leq a_{n,d} + b_{n,d}\, x \right).
 \end{align*}
\begin{thm}\label{T:convergence}
Let \eqref{ass-main} hold. Then, for any fixed $x,x' \in \Reals$, we have
\[
      F_{n,d}(x,x') -  \hat{F}_{n,d}(x,x')   = o\left(\dfrac{\log^2 n}{n^{1/2}}\right).
\]
Furthermore, the marginal distributions   
$F_{n,d}(x,-\infty)$
and $F_{n,d}(\infty,x')$ satisfy the same bound.
\end{thm}
Our plan for this section is as follows. First, estimating  $ \hat{F}_{n,d}(x,x')$, we 
derive Theorem~\ref{T:main} from  Theorem \ref{T:convergence}. Then, combining  Theorem \ref{t:phi} and Theorem~\ref{T:extremal}, we  prove 
Theorem \ref{T:convergence}. 

\subsection{Proof of Theorem \ref{T:main}}\label{S:main}

 From Theorem \ref{T:convergence}, we get that, for any fixed $x,x'\in \Reals$
\begin{align*}
   \Pr\Big(&\dfrac{X_{\max}(n,d) -a_{n,d}}{b_{n,d}} \leq x,\dfrac{2\lambda^2n-a_{n,d}-X_{\min}(n,d) }{b_{n,d}} \leq -x'\Big) 
   \\ &=    F_{n,d}(x,x') = \hat{F}_{n,d}(x,x') + o(1).
\end{align*}
Using the lemma  below, it is straightforward to show that 
\[
\hat{F}_{n,d}(x,x') = \Big(\Pr\left( 2\lambda^2 n -a_{n,d} + b_{n,d}\, x'  \leq X_{ij}(n,d) \leq a_{n,d} + b_{n,d}\,x \right)\Big)^{\binom{n}{2}} 
\longrightarrow e^{-e^{-x} - e^{x'}},
\]
thus completing the proof of Theorem \ref{T:main}.

\begin{lemma}
Let \eqref{ass-main} hold. Then, for any fixed $x,x' \in \Reals$, we have
\begin{align*}
    &\Pr\Big(X_{ij}(n,d) >  a_{n,d} + b_{n,d} x\Big)  \sim
    \dfrac{1}{\binom{n}{2}} e^{-x},\\
   &\Pr\Big(X_{ij}(n,d) < 2\lambda^2 n -a_{n,d} + b_{n,d}\, x' \Big)  \sim
    \dfrac{1}{\binom{n}{2}} e^{x'}.
\end{align*}
\label{cl:prob_single_event}
\end{lemma}
\begin{proof}
Set 
   \[
 	  			N := \left\lfloor \frac{\lambda}{2-\lambda}n \right\rfloor, \qquad p:=\lambda(2-\lambda), \qquad m:=\binom{n}{2}.
 	  \]     
Let  $\xi$ be  distributed according to $\Bin(N,p)$.
Using Theorem \ref{T:bin}, we get that
 \[
 		\left(1 - \Pr\left(  \xi>  a_n^*  + b_n^*  x \right)\right)^m  
 		 \rightarrow e^{-e^{-x}},
 \]
where $a_n^* =a_n^*(N,m,p)$, $b_n^*=b_n^*(N,m,p)$ are defined in~\eqref{def:ab}. Thus,
\[
 		\Pr\left(  \xi>  a_n^*  + b_n^*  x \right) \sim  \dfrac{1}{m} e^{-x}.
\]
From definition \eqref{def:ab}, 
we find that 
\[
a_{n,d}   =  a_n^* +o(b_n) \quad \text{ and } \quad
b_{n,d}  \sim b_n^* \rightarrow \infty.
\] Therefore,
$$
\Pr\left(  \xi>  a_{n,d}  + b_{n,d}  x \right) \sim  \dfrac{1}{m} e^{-x}.
$$
Applying the Chernoff bound, we get that with probability at least $1 - e^{-\omega(\log n)}$, 
\begin{equation}\label{eq:Chernoff}
 | \xi-\lambda^2 n |   \leq \sqrt{n} \log n.
\end{equation}
 Using also Corollary~\ref{L:dist}, we find that 
 \begin{equation}\label{eq:part_a}
 		 \Pr\left( X_{ij}(n,d) >  a_{n,d} + b_{n,d} \, x\right)
 		 	=(1+o(1)) \Pr\left(  \xi>  a_{n,d}  + b_{n,d}\,  x \right)  +   e^{-\omega(\log n)} 
 		 	\sim   \dfrac{1}{m} e^{-x}
 \end{equation}
 completing the proof of the first claimed bound.

The proof of the second bound is similar by applying Theorem \ref{T:bin} for $N-\xi\sim \Bin(N,1-p)$.
Additionally, we only need to observe the symmetry relation
\[
		a_n^*(N,m,1-p) =  a_n^*(N,m,p) + (1-2p)N,\qquad  b_n^*(N,m,1-p) = b_n^*(N,m,p),
\]
which  gives
\[
	\Pr\left(  \xi < 2 \lambda^2n  - a_{n,d}  + b_{n,d}\,   x'\right) 
	\sim \Pr \Big( N-\xi  > a_n^*(N,1-p)  -   b_n^*(N,1-p) x'  \Big)  \sim \dfrac{1}{m} e^{x'}.
\]
Using Corollary~\ref{L:dist}, we derive the second bound.
\end{proof}

\subsection{Proof of Theorem \ref{T:convergence}}
For $ij \in \binom{[n]}{{2}}$, consider the events $A_{ij}$  defined by
\begin{align*}
& A_{ij}:=\{X_{ij}(n,d) \in I^+(x) \cup I^-(x') \},\\ 
         &I^+(x):=
         (a_{n,d} + b_{n,d}\, x, \lambda^2 n + \sqrt{n} \log n),
         \\
            &I^-(x'):= (\lambda^2 n - \sqrt{n} \log n, 2\lambda^2 n  -a_{n,d} + b_{n,d}\, x').
            \end{align*}
Using Corollary \ref{L:dist}(b), we get that 
\begin{equation}\label{eq:F-F}
F_{n,d}(x,x') -  \hat{F}_{n,d}(x,x')=\Pr\left(\bigcap_{ ij \in \binom{[n]}{2}}\overline{A_{ij}}\right)  -  \prod_{  ij \in \binom{[n]}{2}}  \Pr\left( \overline{A_{ij}}\right) +  e^{-\omega(\log n)}.
\end{equation}
To estimate the RHS of \eqref{eq:F-F}, we apply Theorem~\ref{T:extremal} for   
$
\A := (A_{ij})_{ij \in \binom{[n]}{2}}.
$
Define the graph $\D$ on ${[n]\choose 2}$ in the following way: for a pair of distinct vertices $i,j \in [n]$, the set $D_{ij}$  consists of pairs that overlap with $ij$, but not coincide. That is, we have
\[
D_{ij} =\left\{i'j' \in \binom{[n]}{2} \st    |\{i,j\} \cap \{i',j'\}| = 1 \right\}.
 \]
Let
$$
\varphi=\varphi(\A,\D),\quad \Delta_1=\Delta_1(\A,\D),\quad
\Delta_2=\Delta_2(\A,\D).
$$

Let $h $ be any integer from  $I^- \cup I^+$, 
that is, $\{X_{12}(n,d)=h\} \subset A_{12}$ and 
let  $U \subseteq \binom{[n] \setminus\{1,2\}}{2}$. To bound $\varphi$, we  show that 
\begin{equation}\label{eq:forphi}
\Pr\left(\bigcup_{ij \in U   } A_{ij}   \mid  X_{12}(n,d)=h\right)
- \Pr\left(\bigcup_{ij \in U     } A_{ij}   \right)   = o\left(\frac{\log^2 n}{n^{1/2}}\right)
\end{equation}
 uniformly over such $U$ and $h$.   
 Consider the coupling $(\calG_{n,d}, \calG_{n,d}^h)$ provided by Theorem \ref{t:phi}.
 Since, with probability $1 -o\left(\dfrac{\log^2 n}{\sqrt{n}}\right)$,
  		the neighbourhoods of each  vertex $k \in [n]\setminus\{1,2\}$ in graphs $\calG_{n,d}$ and  $\calG_{n,d}^h$ differ by at most $8$ elements,  using the union bound,
    we get that  
    \begin{align*}
        \Big|\Pr\Big( \bigcup_{ij \in U   } A_{ij} &\mid  X_{12}(n,d)=h \Big) -  
        \Pr\Big(\bigcup_{ij \in U   } 
 A_{ij}\Big)\Big| 
 \\ &\leq  
         \sum_{ij\in U}\Pr\Big(X_{ij}(n,d) \in I_1 \cup I_2 \cup I_3 \cup I_4\Big) +o\left(\dfrac{\log^2 n}{\sqrt{n}}\right),
        \end{align*}
 where $I_1, I_2, I_3, I_4$ are the balls of radius  16 around the endpoints of $I^+(x)$ and $I^-(x)$. Using Corollary \ref{L:dist}(a), we get that
 \[
    \Pr(X_{ij}(n,d) \in I_1 \cup I_2 \cup I_3 \cup I_4) = O\left( \max_{s} \,\Pr(\xi =s)\right) = O(n^{-5/2}),
 \]
where $\xi \in \Bin(N,p)$ with $N =  \left\lfloor \dfrac{\lambda}{2-\lambda}n \right\rfloor$ and  $p=\lambda(2-\lambda) $. The claimed bound \eqref{eq:forphi} follows. Considering similar bounds for all other pairs of vertices instead of $1$  and $2$, we show   $\varphi = o\left(\dfrac{\log^2 n}{\sqrt{n}}\right)$.

Next, we estimate $\Delta_1$ and $\Delta_2$. 
From Lemma~\ref{cl:prob_single_event}, we know that 
$\Pr(A_{i,j})=O(n^{-2})$.   Since the number of edges in $\D$ is $O(n^3)$, we get that $\Delta_2=O(n^{-1})$.
From Lemma \ref{L:eps-gamma} stated below it follows that 
\begin{equation*}
  \Delta_1 = O(\Delta_2)=O(n^{-1}).
\end{equation*}
Then, applying Theorem \ref{T:extremal}, we get  Theorem \ref{T:convergence}.

\begin{remark}
The proof of the bounds for marginal distributions 
$F_{n,d}(x,\infty)$
and $F_{n,d}(\infty,x')$ 
follows exactly the same steps with the only modification: we ignore  $I^-(x')$ or $I^+(x')$ in the definition of the events $A_{ij}$. 
\end{remark}

\begin{lemma}\label{L:eps-gamma}
 Let $\varepsilon  \in \left(0, \frac{1}{2}\min\{\lambda^2,(1-\lambda)^2\}\right) $ be fixed. There exists a constant $\gamma>0$ such that
\[
\Pr\Big(X_{ij}(n,d)\in Y \text{ and } X_{ij'}(n,d)\in Y'\Big)\leq\gamma\,\Pr(X_{ij}(n,d)\in Y)\cdot\Pr(X_{i,j'}(n,d)\in Y').
\]
  for all distinct $i,j,j'\in[n]$ and any integer   sets
  $Y,Y'\subset [(\lambda^2 -\eps) n, (\lambda^2+\varepsilon)n]$. 
\label{cl:Delta1_from_Delta_2}
\end{lemma}
\begin{proof} 
Without loss of generality, we let $i=1$, $j=2$, $j'=3$. Clearly, it is sufficient to prove the lemma for singleton sets $Y=\{y\}$ and $Y'=\{y'\}$, where $y,y'\in[(\lambda^2-\varepsilon)n,(\lambda^2+\varepsilon)n]$. Let us fix three {\it consistent} sets $U_1\subset[n]\setminus\{1\}$, $U_2\subset[n]\setminus\{2\}$, $U_3\subset[n]\setminus\{3\}$ of size $d$. By consistency we mean that $2\in U_1$ if and only if $1\in U_2$, and the same holds true for other pairs of sets. In other words, the event $\{N_1(\calG_{n,d})=U_1,\,N_2(\calG_{n,d})=U_2,\,N_3(\calG_{n,d})=U_3\}$ has positive probability. Note that, subject to $\{N_1(\calG_{n,d})=U_1,\,N_2(\calG_{n,d})=U_2\}$, the random graph $\calG_{n,d}$ has uniform distribution over the set of all $d$-regular graphs on $[n]$ satisfying the condition. Since the neighbourhoods of vertices $1,2$ of this graph are determined, it can be treated as a uniform random graph on $\{3,\ldots,n\}$ with almost $d$-regular degree sequence. By Lemma~\ref{l:neighb},
\begin{equation}
 \Pr(N_3(\calG_{n,d})=U_3\mid N_1(\calG_{n,d})=U_1,\,N_2(\calG_{n,d})=U_2)=\Theta\left(\sqrt{n}\lambda^d(1-\lambda)^{n-d}\right)
   \label{eq:cond_two_sets}
\end{equation}
uniformly over the choice of $U_1,U_2,U_3$. In the same way,
\begin{equation}
 \Pr(N_3(\calG_{n,d})=U_3\mid N_1(\calG_{n,d})=U_1)=\Theta\left(\sqrt{n}\lambda^d(1-\lambda)^{n-d}\right)
   \label{eq:cond_one_set}
\end{equation}
uniformly over the choice of $U_1,U_3$.
Due to \eqref{eq:cond_two_sets},~\eqref{eq:cond_one_set},   for some sufficiently large constant $\gamma_1>0$ (independent of $U_1,U_2,U_3$),
\begin{equation}
\begin{aligned}
  \Pr\(N_3(\calG_{n,d})=U_3&\mid N_1(\calG_{n,d})=U_1,\,N_2(\calG_{n,d})=U_2\) \\ &\leq \gamma_1
  \Pr\(N_3(\calG_{n,d})=U_3\mid N_1(\calG_{n,d})=U_1\).
  \end{aligned}
\label{eq:gamma_1}
\end{equation}

Next, we fix some $y,y'\in[(\lambda^2-\varepsilon)n,(\lambda^2+\varepsilon)n]$. For $U_1\in{[n]\setminus\{1\}\choose d}$, let 
$$
\mathcal{U}_3(U_1):=\left\{U_3\in{[n]\setminus\{3\}\choose d} \st 
 \text{$U_3$ is consistent with $U_1$ and } |U_1\cap U_3|=y'\right\}.
$$
Note that 
\[
    |\mathcal{U}_3(U_1)|= \begin{cases}
                                {d-1\choose y'}{n-1-d\choose d-1-y'}, & \text{if $3\in U_1$},\\
                                {d\choose y'}{n-2-d\choose d-y'}, & \text{otherwise.} 
                            \end{cases}
\]
Let $P_3(U_1)$ denote the conditional  probability that the number of common neighbours of $1$ and $3$  in $\calG_{n,d}$  equals $y'$  given  that $N_1(\calG_{n,d}) = U_1$, that is,
\[
P_3(U_1):=\sum_{U_3\in\mathcal{U}_3(U_1)}\Pr\(N_3(\calG_{n,d})=U_3\mid N_1(\calG_{n,d})=U_1\). 
\]
Due to symmetry,  $P_3(U_1)$ takes only two values $P_3^{in}$ and $P_3^{out}$ depending on whether $3 \in U_1$ or not.
Since $\varepsilon  \in \left(0, \frac{1}{2}\min\{\lambda^2,(1-\lambda)^2\}\right) $, we get that $\varepsilon<\lfloor\lambda-\lambda^2\rfloor$. Thus,
$$
\frac{{d-1\choose y'}{n-1-d\choose d-1-y'}}{{d\choose y'}{n-2-d\choose d-y'}}=\frac{(d-y')^2(n-1-d)}{d(n-2d+y')(n-2d+y'-1)}=\Theta(1).
$$
Then, using \eqref{eq:cond_one_set},
    there is  a sufficiently small constant $\gamma_2>0$ (independent of the choice of $y'$) such that
$$
\min\{P_3^{in},P_3^{out}\}\geq \gamma_2\max\{P_3^{in},P_3^{out}\}.
$$
Also, due to  symmetry,   $\Pr(N_1(\calG_{n,d})=U_1)$ is independent of  the choice of $U_1$. Therefore, 
\begin{align}
 \Pr(X_{13}(n,d)=y') &=
 \sum_{U_1\in {[n]\setminus\{1\}\choose d}}\sum_{U_3\in\mathcal{U}_3(U_1)}
 \Pr(N_1(\calG_{n,d})=U_1,\,N_3(\calG_{n,d})=U_3)\notag\\
 &=
\sum_{U_1}\Pr(N_1(\calG_{n,d})=U_1)\sum_{U_3}\Pr(N_3(\calG_{n,d})=U_3\mid N_1(\calG_{n,d})=U_1)\notag\\
&\geq\min\{P_3^{in},P_3^{out}\}\geq \gamma_2\max\{P_3^{in},P_3^{out}\}. 
\label{eq:gamma_2}
\end{align}

Finally, for $U_1\in{[n]\setminus \{1\}\choose d}$ and $U_3\in\mathcal{U}_3(U_1)$, we let
\begin{align*}
\mathcal{U}'_2(U_1,U_3) & =\left\{U_2\in{[n]\setminus\{2\}\choose d} \st \text{$U_2$ is consistent with $U_1,U_3$ and } |U_1\cap U_2|=y\right\},\\
\mathcal{U}_2(U_1) & =\left\{U_2\in{[n]\setminus\{2\}\choose d} \st 
\text{$U_2$ is consistent with $U_1$ and } |U_1\cap U_2|=y\right\}.
\end{align*}
Using \eqref{eq:gamma_1}, we get that
\begin{align*}
\Pr(& X_{12}(n,d)=y,\, X_{13}(n,d)=y')\\
 &=
\sum_{U_1,U_2,U_3}
 \Pr(N_1(\calG_{n,d})=U_1,\,N_2(\calG_{n,d})=U_2,\,N_3(\calG_{n,d})=U_3)\\
 & \stackrel{\eqref{eq:gamma_1}}\leq \gamma_1
\sum_{U_1,U_2,U_3}
 \Pr(N_3(\calG_{n,d})=U_3\mid N_1(\calG_{n,d})=U_1)\cdot \Pr(N_1(\calG_{n,d})=U_1,\,N_2(\calG_{n,d})=U_2),
\end{align*}
where the both sums above are  over 
$U_1\in {[n]\setminus\{1\}\choose d}$,  $U_3\in\mathcal{U}_3(U_1)$, and
$U_2\in\mathcal{U}'_2(U_1,U_3)$.
Using also \eqref{eq:gamma_2}, we conclude that
\begin{align*}
 \Pr( X_{12}(n,d)\in Y,\, &X_{13}(n,d)\in Y')\\
 & \leq \gamma_1\max\{P_3^{in},P_3^{out}\}\sum_{U_1\in {[n]\setminus\{1\}\choose d},\, U_2\in\mathcal{U}'_2(U_1,U_3)}\Pr(N_1(\calG_{n,d})=U_1,\,N_2(\calG_{n,d})=U_2)\\
  & \leq \gamma_1\max\{P_3^{in},P_3^{out}\}\sum_{U_1\in {[n]\setminus\{1\}\choose d},\, U_2\in\mathcal{U}_2(U_1)}\Pr(N_1(\calG_{n,d})=U_1,\,N_2(\calG_{n,d})=U_2)\\
 &= \gamma_1\max\{P_3^{in},P_3^{out}\}\Pr(X_{12}(n,d)=y)\\
 &\stackrel{\eqref{eq:gamma_2}}\leq
 \frac{\gamma_1}{\gamma_2}\Pr(X_{12}(n,d)=y)\cdot\Pr(X_{13}(n,d)=y'). 
\end{align*}
 Letting $\gamma=\gamma_1/\gamma_2$, we complete the proof.
\end{proof}

\section{Local limit theorem for common neighbours}\label{S:distribution}

In this section,  we  prove   Theorem \ref{T:distribution_pair} and then get   Corollary \ref{L:dist} as a consequence. 
Recall that $N_i(G)$ denotes the set of neighbours of the vertex $i$ in a graph $G$.

 \subsection{Proof of Theorem \ref{T:distribution_pair}}
 Let $A,B \subset [n]\setminus\{i,j\}$, $|A|=|B|=d$, and $|A\cap B|=h$.
  First, we compute the probability that
 $A, B$ are the sets of neighbours of $i$ and $j$ in $\calG_{n,d}$.   
 \begin{equation}\label{eq:iAjB}
 \begin{aligned}
 	&\Pr\Big( N_i(\calG_{n,d}) = A, \
   N_j(\calG_{n,d}) = B\Big)   \\
   &\hspace{2cm}= \Pr\Big(N_i(\calG_{n,d}) = A\Big) \cdot  \Pr\Big(N_j(\calG_{n,d}) = B  \mid  N_i(\calG_{n,d}) = A\Big).
   \end{aligned}
 \end{equation}
 From  Lemma \ref{l:neighb}, we know that 
 \begin{align*}
 	\Pr(N_i(\calG_{n,d}) = A)  
 	  \sim \sqrt{2\pi n}\, \lambda^{d+\frac12} (1-\lambda)^{n-d-\frac12}.
 \end{align*}
We can also use Lemma \ref{l:neighb} to find the second factor  in the right-hand side  of \eqref{eq:iAjB}.
 Indeed, conditioning with respect to the neighbourhood of vertex $i$ is equivalent to the random graph $\calG_{\dvec'}$ with almost $d$-regular degree sequence 
 $\dvec' \in \Naturals^{n-1}$.
 Applying  Lemma~\ref{l:neighb}, we find that
 \begin{align*}
 	  &\Pr\Big(N_j(\calG_{n,d}) = B  \mid  N_i(\calG_{n,d}) = A\Big)
 	  \\  &\sim  \sqrt{2\pi \lambda   (1-\lambda ) n}  
 	  \cdot
 	  \prod_{k \in B} \dfrac{ d - \mathbf{1}_{k\in A}}{n-2}  \prod_{k \notin  B \cup\{i,j\} } \left(1 -  \dfrac{d - \mathbf{1}_{k\in A}}{n-2}\right) 
 	  \\
 	    &\sim   \sqrt{2\pi  n} \, \lambda^{d+\frac12} (1-\lambda)^{n-d-\frac32} \left(\dfrac{n-1}{n-2}\right)^{n-2}  \left(\dfrac{d-1}{d}\right)^{h}   \left(\dfrac{n-d-2}{n-d-1}\right)^{n-2-2d+h}.
\end{align*}
Observe that
\begin{align*} 	 
\left(\dfrac{n-1}{n-2}\right)^{n-2}         \left(\dfrac{n-d-2}{n-d-1}\right)^{n-2-2d}
 \sim    
 	       \left(\dfrac{n-d-2}{n-d-1}\right)^{-d}    
 	     &\sim	   \exp\left(\dfrac{\lambda}{1-\lambda} \right),
\\	 
    \left(\dfrac{d-1}{d}\right)^{h}        \left(\dfrac{n-d-2}{n-d-1}\right)^{h}
 	      \sim	   \exp\left(  -   \dfrac{h}{d} - \dfrac{h}{n-d-1} \right) &\sim   \exp\left(  -   \dfrac{h} {\lambda(1-\lambda)n}\right).
 \end{align*}
  Substituting the above formulas   
 into \eqref{eq:iAjB}, we derive  that
 \[
 	\Pr\Big( N_i(\calG_{n,d}) = A, \
   N_j(\calG_{n,d}) = B\Big) \sim   2\pi n \,\lambda^{2d+1} (1-\lambda)^{2n-2d-2}
 	 \exp\left(\dfrac{\lambda}{1-\lambda} -   \dfrac{h} {\lambda(1-\lambda)n}\right).
 \]
This formula can be rewritten  as 
 \begin{equation}\label{eq:NN}
 	\Pr\Big( N_i(\calG_{n,d}) = A, \
   N_j(\calG_{n,d}) = B\Big) \sim  (1-\lambda)  \binom{n-2}{d}^{-2} \exp\left(\dfrac{\lambda}{1-\lambda} -   \dfrac{h} {\lambda(1-\lambda)n}\right),
 \end{equation}
by  using the Stirling approximation to estimate 
$ 
 	\binom{n-2}{d}
  \sim
 	  \sqrt{\frac{1}{2\pi n}}  
 	   \lambda^{-d-\frac12 }(1-\lambda)^{-n+ \frac32 +d}.
$

Next, the number of   choices of  $A,B \subset [n]\setminus\{i,j\}$ such that    $|A|=|B|=d$ and $|A\cap B|=h$
 equals   $\binom{n-2}{d} \binom{d}{h} \binom{n-2-d}{d-h}$. 
Summing \eqref{eq:NN} over all such choices,  we get that
  \[
 	\Pr\Big(|N_{i}(\calG_{n,d}) \cap N_{j}(\calG_{n,d})| = h,\  ij \notin \calG_{n,d}\Big) \sim  
 	 (1-\lambda)\frac{ \binom{d}{h}  \binom{n-2-d}{d-h} } {\binom{n-2}{d}}   \exp\left(\dfrac{\lambda}{1-\lambda} -   \dfrac{h} {\lambda(1-\lambda)n}\right).
 \]
From  \cite[Theorem 4]{McKay2011} we know that   
 $\Pr(ij \notin \calG_{n,d}) \sim 1-\lambda$.
 Part (a) follows.
 
 The proof of part (b) is similar to part (a). The only difference is that we need to consider 
 the sets $A \subset [n]\setminus\{i\}$ and $ B \subset [n]\setminus\{j\}$ such that  
 $i \in B$ and $j \in A$. First,  for such $A,B$, applying  Lemma \ref{l:neighb}
 and using the Stirling approximation, we find that  
 \begin{equation*} 
 	\Pr\Big(N_i(\calG_{n,d}) = A,\ N_j(\calG_{n,d}) = B\Big) \sim  \lambda  \binom{n-2}{d-1}^{-2} \exp\left(\dfrac{\lambda}{1-\lambda} -   \dfrac{h} {\lambda(1-\lambda)n}\right).
 \end{equation*}
 Summing  over all   choices
 $A  \subset [n]\setminus\{i\}$, $B \subset [n]\setminus\{j\}$ such that    $|A|=|B|=d$,  $|A\cap B|=h$,
 and $i \in B$, $j \in A$,  we get that
  \[
 	\Pr\Big(|N_{i}(\calG_{n,d}) \cap N_{j}(\calG_{n,d})| = h,\  ij \in \calG_{n,d}\Big)  \sim  
 	 \lambda\frac{ \binom{d-1}{h}  \binom{n-1-d}{d-1-h} } {\binom{n-2}{d-1}}   \exp\left(\dfrac{\lambda}{1-\lambda} -   \dfrac{h} {\lambda(1-\lambda)n}\right).
 \]
 From  \cite[Theorem 4]{McKay2011} we know that  
 $\Pr(ij \in \calG_{n,d}) \sim \lambda$.
 Part (b) follows.

 \subsection{Proof of Corollary \ref{L:dist}}
 
By assumptions of part (a), we  have that 
 	 \begin{equation*}
 	    \dfrac{\lambda}{1-\lambda} -   \dfrac{h} {\lambda(1-\lambda)n}  = o(1).
 	 \end{equation*}
Applying  Theorem \ref{T:distribution_pair}, we find that
 	 \begin{align*}
 	 	\Pr\Big(X_{ij}(n,d) = h \mid  ij \notin \calG_{n,d}\Big)  &\sim 
 	 	\frac{ \binom{d}{h}  \binom{n-2-d}{d-h} } {\binom{n-2}{d}}   
 	 	= \dfrac{(n-1-2d+h)(n-1)}{(n-1-d)^2}  \cdot \frac{\binom{d}{h} \binom{n-1-d}{d-h}}{ \binom{n-1}{d}},
		 \\
		 \Pr\Big(X_{ij}(n,d) = h \mid  ij \in \calG_{n,d}\Big)   &\sim 
 	 	\frac{ \binom{d-1}{h}  \binom{n-1-d}{d-1-h} } {\binom{n-2}{d-1}}   
 	 	 = \dfrac{(d-h)^2(n-1)}{d^2 (n-2d+h)} \cdot 
 	 	 \frac{\binom{d}{h} \binom{n-1-d}{d-h}}{ \binom{n-1}{d}}.
 	 \end{align*}
 	 Recalling that 
 	 \[
 	 		d \sim \lambda n, \qquad  d-h \sim (1-\lambda)\lambda n, \qquad 
 	 	n-2d+h \sim (1-\lambda)^2 n,
 	 \]
 	   and using the law of total  probability,  we get that
 	  	\[
 	  		\Pr\Big(X_{ij}(n,d) = h\Big)  	
 	  		    \sim  	\frac{\binom{d}{h} \binom{n-1-d}{d-h}}{ \binom{n-1}{d} }.
 	  	\]

 	  	Next, let  $t := h - \lambda d$.    By the assumptions, we get  $t =  O(\sqrt{n}\log n )$.
 	 Using the de Moivre--Laplace theorem, we find that
 \[
 	 	\binom{d}{h} \lambda^{h} (1-\lambda)^{d-h} \sim 
 	 	\frac{1}{\sqrt{2\pi   \lambda(1-\lambda) d}} e^{- \frac{(h-\lambda d)^2}{ 2  \lambda(1-\lambda)d }} \sim
 	 	\frac{1}{\sqrt{ 2\pi   \lambda^2 (1-\lambda) n}} e^{- \frac{t^2}{ 2  \lambda^2(1-\lambda)n }}.
\]
and 
\begin{align*}
	\binom{n-1-d}{d-h} \lambda^{d-h}(1-\lambda)^{n-1-2d+h} & \sim 
\frac{1}{\sqrt{2\pi   \lambda(1-\lambda) (n-1-d)}} e^{- \frac{(d-h-\lambda (n-1-d))^2}{ 2  \lambda(1-\lambda) (n-1-d) }}
\\ &\sim \frac{1}{\sqrt{ 2\pi   \lambda (1-\lambda)^2 n}} e^{- \frac{t^2}{ 2  \lambda (1-\lambda)^2n }}.
\end{align*}
Similarly, we get also  that
\begin{align*}
		 \binom{n-1}{d} \lambda^{d} (1-\lambda)^{n-1-d}& \sim 
		 \frac{1}{\sqrt{2\pi   \lambda(1-\lambda) n} }; \\
		 \binom{N}{h} p^h (1-p)^{N-h} &\sim 
		 	\frac{1}{\sqrt{2\pi   p(1-p) N}} e^{- \frac{(h-pN)^2}{ 2  p(1-p)N }}
		 	\sim   \frac{1}{\lambda(1-\lambda) \sqrt{2\pi    n}} e^{- \frac{t^2}{ 2  \lambda^2 (1-\lambda)^2n }}.
\end{align*}
%
Combining the above, we derive that 
\begin{equation*} 
		\frac{\binom{d}{h} \binom{n-1-d}{d-h}}{ \binom{n-1}{d}  }
		 \sim    \frac{1}{\lambda(1-\lambda) \sqrt{2\pi n}}  e^{- \frac{t^2}{ 2  \lambda^2(1-\lambda)  n }-  \frac{t^2}{ 2  \lambda(1-\lambda)^2  n }} 
		 \sim    \binom{N}{h} p^h (1-p)^{N-h}.
\end{equation*}
This completes the proof of part (a).

We proceed to part  (b). 
Since asymptotic bounds in Theorem \ref{T:distribution_pair} hold uniformly over $h\in[d]$ and the factor $\exp\left(\frac{\lambda}{1-\lambda}-\frac{h}{\lambda(1-\lambda)n}\right)$ is bounded for all $h\in[d]$, it is sufficient to prove that
$$
 \sum_{|h-\lambda^2n|>\sqrt{n}\log n}\frac{ \binom{d}{h}  \binom{n-2-d}{d-h} } {\binom{n-2}{d}}=e^{-\omega(\log n)},\quad
 \sum_{|h-\lambda^2n|>\sqrt{n}\log n}\frac{ \binom{d-1}{h}  \binom{n-1-d}{d-h-1} } {\binom{n-2}{d-1}}=e^{-\omega(\log n)}.
$$
Now, part (b) follows from exponential tail bounds for hypergeometric random variables; see, for example, \cite{GW2017}.

\section{Conditioning  with respect to the number  of  common neighbours}\label{S:cond}

In this section we prove Theorem \ref{t:phi}.  
For $h \in [d]$, let $S_{n,d}^{h}$ denote the set  of $d$-regular graphs with vertex set $[n]$ such that vertices $i$ and $j$ have exactly $h$ common neighbours.  Let $\calG_{n,d}^h$ denote the uniform random element of $S_{n,d}^h$.  First, we construct a bipartite meta-graph, whose vertices are graphs of
$S_{n,d}^h$ and $S_{n,d}^{h+1}$.  
Using  a general coupling theorem for  bipartite graphs, we get
  a coupling $(\calG_{n,d}^h, \calG_{n,d}^{h+1})$  that does not change much the graph structure. 
Then, we combine several couplings $(\calG_{n,d}^h, \calG_{n,d}^{h+1})$ to get the desired coupling  of  $\calG_{n,d}$ and  $\calG_{n,d}^h$.

\subsection{General coupling for a bipartite graph}\label{S:bipartite}

In this section, we establish a coupling result   in a general setting, which we later use to prove Theorem \ref{t:phi}.
Let $D$ be a bipartite graph with vertices partitioned into sets  $S$ and $T$. For simplicity we identify $D$ with its set of edges from $S \times T$.
The first part of the following theorem appeared as \cite[Theorem 2.1]{IMSZ} with slightly better constants in the estimate for  $\Pr(XY \notin D)$. 
\begin{thm}\label{T:coupling}
	Let $\delta,\eps \in (0,1)$ and
	\begin{align*}
		\GoodS &:= \left\{x\in S \st  \deg_D(x) \geq \dfrac{(1-\eps)|D|}{|S|}\right\},\\
		\GoodT &:= \left\{y\in T \st  \deg_D(y) \geq \dfrac{(1-\eps)|D|}{|T|}\right\}.
	\end{align*}
	Assume that $|\GoodS| \geq (1-\delta)|S|$ and  $|\GoodT| \geq (1-\delta)|T|$.
			Then, there is a coupling $(X,Y)$ such that
			$X, Y$ are uniformly distributed on $S$ and $T$, respectively, and
			\[
			\Pr(XY \notin D) \leq   2 \eps +  4\delta.
			\]
			Furthermore,  for   any set of edges $H \subseteq D$,  
			\[
			\max_{x\in S}\,\Pr\left(XY \in H \mid X=x\right) \leq \varDelta_S(H) \left(\dfrac{|S|}{|D|} + \dfrac{2}{(1-\delta)|T|} \right),
			\] 
			where $\varDelta_S(H)$  is the maximal number of edges in $H$ incident to  a vertex from $S$. 
		\end{thm}

		\begin{proof}
			First, we construct a coupling that  produces $ \tilde{X} \in \GoodS$ and $\tilde{Y}\in \GoodT$.
			All random variables in the following procedure are generated independently.

			\begin{itemize}
				\item[1.] Uniformly at random choose an edge $\hat{X}\hat{Y} \in D$.
				\item[2.] Uniformly at random choose vertices  $X'\in \GoodS$ and  $Y'\in \GoodT$.
				\item[3.]   If $\hat{X}\notin  \GoodS$ then set $\tilde{X}:=X'$.  
				\item [4.]  If $\hat{Y}\notin  \GoodT$ then set $\tilde{Y}:=Y'$.
				\item[5.]   If $\hat{X}\in \GoodS$  generate
				$\xi_X \in \operatorname{Bernoulli}\left(\dfrac{(1-\eps)|D|}{|S|\deg_D(\hat{X})}\right)$. 
				\item[6.]   If $\hat{Y}\in  \GoodT$  generate $\xi_Y \in \operatorname{Bernoulli}\left(\dfrac{(1-\eps)|D|}{|T|\deg_D(\hat{Y})}\right)$.
				\item[7.]  Set 
				$
				\tilde{X}:=    \begin{cases}
					\hat{X}, & \text{if $\xi_X =1$},\\
					X', & \text{otherwise},
				\end{cases}
				$  \quad
				$
				\tilde{Y}:=    \begin{cases}
					\hat{Y}, & \text{if $\xi_Y =1$},\\
					Y', & \text{otherwise}.
				\end{cases}
				$
			\end{itemize}
			
			\noindent
			For any $x\in  \GoodS$,  observe that
			\begin{align*}
				\Pr(\tilde{X} = x) &= 
				\Pr(\xi_X = 1 \text{  and } \hat{X} =x)\\
				&+\Pr(X' = x)  \left( \Pr( \hat{X} \notin \GoodS) +
				\Pr(\xi_X = 0  \text{  and } \hat{X} \in \GoodS)     	\right). 
			\end{align*}
			Clearly,   $ \Pr(X' = x) = \frac{1}{|\GoodS|} $.
			Since $\Pr(\hat{X}=x) = \frac{\deg_D(x)}{|D|}$,  we derive that 
			\begin{equation}\label{eq:1-eps}
				\Pr(\xi_X = 1 \text{  and } \hat{X} =x) =  \Pr(\xi_X = 1 \mid \hat{X} =x)  \dfrac{\deg_D(x)}{|D|} = \dfrac{1-\eps}{|S|}.
			\end{equation}
			Thus,  $\Pr(\tilde{X} = x)$ is independent of $x$, that is, $\tilde{X}$ is uniformly distributed on $\GoodS$.
			Similarly, we show that  $\tilde{Y}$ is uniformly distributed on $\GoodT$.

			Next, by the construction, if  $\hat{X} \in \GoodS$, $\hat{Y}\in \GoodT$, and $\xi_X = \xi_Y = 1 $ then  
			$\tilde{X}\tilde{Y} \in D$.
			Combining \eqref{eq:1-eps} (and the same formula for $ \Pr(\xi_Y = 1 \text{  and } \hat{Y} =y) $) and the union bound, we find that
			\begin{align*}
				\Pr(\tilde{X}\tilde{Y} \notin D)&\leq 1 -    \Pr(\xi_X = 1 \text{  and } \hat{X}  \in \GoodS)
				+1 - \Pr(\xi_Y = 1 \text{  and } \hat{Y}  \in \GoodS)
				\\
				&= 2 - \dfrac{(1-\eps) |\GoodS|}{|S|} -   \dfrac{(1-\eps) |\GoodT|}{|T|} 
				\leq 2 - 2(1-\eps)(1-\delta) \leq 2\eps + 2\delta.
			\end{align*}
			
			To complete the construction of $X$ and $Y$, we consider $X''$  generated uniformly from $S - \GoodS$.
			Set $X := \tilde{X}$ with probability  $| \GoodS|/  |S|$ and $X:= X''$ 
			with probability  $1-  | \GoodS|/  |S|$. Similarly, define $Y$. Then, 
			$X$, $Y$ are uniformly distributed on $S$ and $T$, respectively. Using the assumptions, we get the required bound
			\[
			\Pr(XY\notin D) \leq \Pr(\tilde{X}\tilde{Y} \notin D) + \Pr(X =X'') + \Pr(Y = Y'') \leq 2\eps + 4\delta.
			\]

			Finally, consider any    $H \subseteq D$. It is sufficient to prove that, for all $x y\in D$,
			\begin{equation}\label{Pr_joint}
				\Pr(X=x, Y=y)  \leq \dfrac{1}{|D|} + \dfrac{2}{(1-\delta)|S|\cdot |T|}.
			\end{equation}
			Indeed, if \eqref{Pr_joint} is true then
			\[
			\Pr(XY \in H \mid X=x)   =   |S|\cdot\sum_{y:\, xy \in H} 
			\Pr(X =x, Y=y)  \leq 
			\varDelta_S(H) \left(\dfrac{|S|}{|D|} + \dfrac{2}{(1-\delta)|T|} \right).
			\]
			If $x\notin \GoodS$  or $y \notin \GoodT$ then the events $X=x$ and $Y=y$ are independent, therefore
			\[
			\Pr(X=x, Y=y) = \Pr(X = x)  \Pr(Y =y) = \dfrac{1}{|S| \cdot|T|},
			\]
			If $x\in \GoodS$  and $y \in \GoodT$, we estimate
			\begin{align*}
			\Pr(X=x, Y=y)  
   &\leq \Pr(\hat{X} = x, \hat{Y}=y) + 
   \Pr(X' = x, Y=y) +  \Pr(X = x, Y'=y)
    \\
    &= \dfrac{1}{|D|} + \dfrac{1}{|\GoodS|} \cdot \dfrac{1}{|T|}
    +\dfrac{1}{|S|} \cdot \dfrac{1}{|\GoodT|}.
			\end{align*}
			The above two bounds imply \eqref{Pr_joint}, completing the proof of the theorem.
		\end{proof}


\subsection{Coupling  of  $\calG_{n,d}^h$ and $\calG_{n,d}^{h+1}$ }\label{s:2coupling}

Recall that $S_{n,d}^{h}$ denotes the set  of $d$-regular graphs with vertex set $[n]$ such that vertices $i$ and $j$ have exactly $h$ common neighbours. 
To apply Theorem \ref{T:coupling}, we construct  the bipartite graph $D$ as follows. 
Let $S:= S_{n,d}^h$ and $T:= S_{n,d}^{h+1}$. Two graphs  $G \in S_{n,d}^h$ and $G' \in S_{n,d}^{h+1}$ are connected by an edge in  $D$    if  there are  distinct vertices $u, v, w \in [n]-\{i,j\}$ such that $iu \in G \cap G'$,
$iv \notin G \cup G'$,
$jv \in G - G'$,
 $uw \in G- G'$, $uj \in G' - G$,
$vw \in G'-G$ and all other edges of $G$ and $G'$ are the same; see Figure~\ref{f:switch} for an illustration.
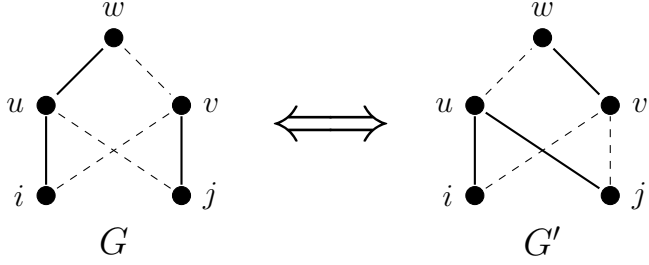
\begin{figure}[h!]
	\centering
	\begin{tikzpicture}[scale=0.6]
		\node (w) [label = above:{$w$}]at (0,4) {};
		\node (u) [label=left:{$u$}]   at (-1.5,2.5) {};
		\node (v)  [label=right:{$v$}]  at (1.5,2.5) {};
		\node (n1) [label=left:{$i$}]  at (-1.5,0.5) {};
		\node (n2) [label=right:{$j$}]  at (1.5,0.5) {};
		\node at (0,-0.5)  {\large $G$};
		

		\draw [-,thick] (n1) -- (u);
		\draw [-,dashed](n2) -- (u);
		\draw[-,thick] (n2) -- (v);
		\draw[-, thick] (u)--(w);
		\draw[-,dashed](n1)--(v);
		\draw[-,dashed](w)--(v);
		
		\draw [fill] (n1) circle (0.2); \draw [fill] (n2) circle (0.2);  \draw [fill] (u) circle (0.2);   \draw [fill] (v) circle (0.2);  \draw [fill] (w) circle (0.2);

		\node   at (4.8,2)  {\Huge $\Longleftrightarrow$};
		
		\begin{scope}[shift={(9.5,0)}]

		\node (w) [label = above:{$w$}]at (0,4) {};
		\node (u) [label=left:{$u$}]   at (-1.5,2.5) {};
		\node (v)  [label=right:{$v$}]  at (1.5,2.5) {};
		\node (n1) [label=left:{$i$}]  at (-1.5,0.5) {};
		\node (n2) [label=right:{$j$}]  at (1.5,0.5) {};
		\node at (0,-0.5)  {\large $G'$};
		

		\draw [-,thick] (n1) -- (u);
		\draw [-,thick](n2) -- (u);
		\draw[-,dashed] (n2) -- (v);
		\draw[-, dashed] (u)--(w);
		\draw[-,dashed](n1)--(v);
		\draw[-,thick](w)--(v);
		
		\draw [fill] (n1) circle (0.2); \draw [fill] (n2) circle (0.2);  \draw [fill] (u) circle (0.2);   \draw [fill] (v) circle (0.2);  \draw [fill] (w) circle (0.2);

		\end{scope}
	\end{tikzpicture}
	\caption{Two graphs $G \in S^{h}_{n,d}$ and $G'\in S^{h+1}_{n,d}$ adjacent in $D$.}
	\label{f:switch}
\end{figure}

\begin{lemma}\label{l:condeg}
  Let \eqref{ass-main} hold.
	  If $h \sim \lambda^2 n$ then 
	\[
		\E[\deg_{D} (\calG_{n,d}^{h})  ]\sim \lambda^3(1-\lambda)^3 n^3 , \qquad \E[\deg_{D} (\calG_{n,d}^{h+1})  ]\sim \lambda^3(1-\lambda)^3 n^3.
	\]
	Furthermore,     with probability at least  $1-  e^{-\omega(\log n)}$,
	  \begin{align*}
	         \deg_{D} (\calG_{n,d}^{h})   &=  \left(1+o\left( \dfrac{\log n}{n}\right)\right)\E[\deg_{D} (\calG_{n,d}^{h})  ], \\
	         \deg_{D} (\calG_{n,d}^{h+1})   &=  \left(1+o\left( \dfrac{\log n}{n}\right)\right)\E[\deg_{D} (\calG_{n,d}^{h+1})  ].
	   \end{align*}  
\end{lemma}
\begin{proof}
	Here,  we only prove the bounds for $\deg_{D} (\calG_{n,d}^{h})$. 
	The proof of the  bounds for $\deg_{D} (\calG_{n,d}^{h+1})$ is entirely similar.

	Consider any $G\in S_{n,d}^{h}$. 
	There are exactly $(d-h)^2$  ways  to choose vertices $u,v$
	such that $iu, jv \in G$ and $iv, ju\notin G$. The number of ways to chose 
	$w \in [n]-\{i,j,u,v\}$ such that $wu \in G$   is $(1+O(n^{-1}))d$. To compute 
	$ \deg_{D} (G) $, we also need to exclude the case when $wv \in G$. Thus, we get that
\begin{equation}\label{eq:deg_D}
		 \deg_{D} (G) =   (1+O(n^{-1})) (d-h)^2  d -    \sum_{u,v} |\{w \in [n] \st uw,vw \in G\}|,
\end{equation}
where the sum is over $u,v$ such that $iu, jv \in G$ and $iv, ju\notin G$.

Let $\calE_{iju}$ be the event that vertices $i,j, u$  have the same neighbourhoods in $\calG_{n,d}^h$ and $G$. Consider the random variable 
$
	\sum_{v} |\{w \in [n] \st uw,vw \in \calG_{n,d}^h\}|,
$ 
 where the sum is over $v$ such that $jv \in \calG_{n,d}^h$ and $iv\notin \calG_{n,d}^h$, 
conditioned to the event $\calE_{iju}$.  Observe that it counts the number of edges between
two sets of linear size (namely, the set of neighbours of $u$ and the set of neighbours  of $j$ not adjacent to $i$) in the uniform random graph on 
vertices $[n]-\{i,j,u\}$ with almost   regular degree sequence (all degrees are $d+O(1)$).   By Lemma~\ref{l:concentration}, this random variable is tightly concentrated near its expectation. Note that some of pairs $vw$ can repeat so it does not immediately follow from Lemma~\ref{l:concentration}. However, in addition to the set of all pairs, we can apply  Lemma~\ref{l:concentration} to either the set of pairs that repeat or to the set of pairs that do not, whichever is bigger.
 Since this concentration holds for all events   $\calE_{iju}$, we find that
\begin{align*}
	\sum_{v} |\{w \in [n] \st uw,vw \in \calG_{n,d}^h\}   
	= \left(1+ o\left( \dfrac{\log n}{n}\right) \right)\E \left[\sum_{v} |\{w \in [n] \st uw,vw \in \calG_{n,d}^h\}|\right]
	\sim \lambda d(d-h)
\end{align*}
with probability at least $1 - e^{-\omega(\log n)}$. 
Summing over $u$ such that $iu \in \calG_{n,d}^h$ and $ju \notin \calG_{n,d}^h$ and using \eqref{eq:deg_D}, we derive that, with the same probability bound,
\[
	 \deg_{D} (\calG_{n,d}^h) = \left(1+ o\left( \dfrac{\log n}{n}\right) \right) \E[ \deg_{D} (\calG_{n,d}^h) ]
	 \sim (1-\lambda) (d-h)^2 d \sim \lambda^3(1-\lambda)^3 n^3.
\]
This completes the proof.
\end{proof}

Note  that if $G$ and $G'$ are adjacent in $D$ then the triple of vertices $(u,v,w)$ is uniquely determined.  We label  such edge $GG'$ of $D$ by $w$.  Let $D^w$ denote the set of edges of $D$ labelled by~$w$.

\begin{lemma}\label{l:2coupling}
			Let \eqref{ass-main} hold and $h \sim \lambda^2 n$. Then,  there is a coupling $(\calG_{n,d}^h, \calG_{n,d}^{h+1})$ such that 
		   \[
		   		\Pr(\calG_{n,d}^h \calG_{n,d}^{h+1}  \notin D) =  o\left( \dfrac{\log n}{n}\right).
		   \]
		   Also, uniformly over  $w \in [n] -\{i,j\}$,   
		   \begin{align*}
		   					\Pr(\calG_{n,d}^h \calG_{n,d}^{h+1}  \in D^w &\mid \calG_{n,d}^{h} )  = O(n^{-1}),\\
		   						\Pr(\calG_{n,d}^h \calG_{n,d}^{h+1}  \in D^w &\mid \calG_{n,d}^{h+1} ) = O(n^{-1}).
		   \end{align*}
	\end{lemma}
\begin{proof}
		By Lemma \ref{l:condeg}, the assumptions of  Theorem \ref{T:coupling} hold with 
		$\eps = o\left( \dfrac{\log n}{n}\right)$ and $\delta = e^{-\omega(\log n)}$. Thus, we get the first part of the lemma.
		
		For the second part, we estimate  $\varDelta_{S}(D^{w}) \leq n^2$ counting all possible pairs of $u,v$.  By Lemma~\ref{l:condeg}, we have 
		 \[ |T|\geq \dfrac{|D|}{|S|} = 	\E[\deg_{D} (\calG_{n,d}^{h})  ]\sim \lambda^3(1-\lambda)^3 n^3.\]
	 Thus, the second part of Theorem \ref{T:coupling} gives
	 \[
	 \Pr(\calG_{n,d}^h \calG_{n,d}^{h+1}  \in D^w \mid \calG_{n,d}^{h} )  = O(n^{-1}).
	 \]
	 The last bound follows by switching the roles of $S$ and $T$ in the definition of $D$.
	\end{proof}

\subsection{Proof of Theorem \ref{t:phi}}


We get the required coupling $(\calG_{n,d}, \calG_{n,d}^h)$ as follows. 
For  all $h'$  that $|h' - \lambda^2 n| \leq  \sqrt{n} \log n$, using  Lemma \ref{l:condeg}, we show the existence of a coupling
$(\calG_{n,d}^{h'}, \calG_{n,d}^{h})$ such that 
the desired event holds with probability $1-o\left(\frac{\log^2 n}{\sqrt{n}}\right)$.
Then, we can glue these couplings taking $h'$ to be a random variable distributed according to $X_{ij}(n,d)$ and generating $\calG_{n,d}^{h}$ independently of $\calG_{n,d}^{h'}$ if $|h' - \lambda^2 n| > \sqrt{n} \log n$. By Corollary~\ref{L:dist}(b), this happens with probability at most $e^{-\omega(\log n)}$.


If $|h' - \lambda^2 n| \leq  \sqrt{n} \log n$, by the assumptions, we find that
\[
	|h' - h |  = O\left(\sqrt{n} \log n \right).
\]
Therefore, we need to combine at most $O(\sqrt{n} \log n )$ couplings from Lemma \ref{l:2coupling}. The probability that there exist some $\hat{h}$ between $h$ and $h'$ such that $\calG_{n,d}^{\hat{h}} \calG_{n,d}^{\hat h+1} \notin D$ is bounded above by 
\[
		 |h'-h|  \cdot	o\left(\dfrac{\log n}{n}\right) = o\left(\dfrac{\log^2 n}{\sqrt{n}}\right).
\]

Next, observe that, according to our construction of $D$ in Section \ref{s:2coupling}, any vertex from $[n] -\{i,j\}$ can play a role of $u$ or $v$  for at most one of $|h-h'|$ switchings. Indeed, if,  for example $h'>h$, then, in order to couple $\calG_{n,d}^{\hat{h}}$ with $\calG_{n,d}^{\hat{h}+1}$ for $h\leq\hat h<h'$, we choose 
$u,v$ adjacent in $\calG_{n,d}^{\hat{h}}$ to exactly one vertex from $\{i,j\}$, while  the corresponding graph
$ \calG_{n,d}^{\hat h+1}$ has $u$ as a common neighbour of $i$ and $j$ and $v$ is adjacent to none of them.
 
The neighbourhood of a vertex is also changed if we use it as the vertex $w$  for  $\calG_{n,d}^{\hat{h}} \calG_{n,d}^{\hat h+1} \in D$. However, by the second part of Lemma \ref{l:2coupling}, the probability that it happens at least $4$ times is bounded above by  
\[
 	|h'-h|^4 \cdot O(n^{-4}) = O\left(\dfrac{\log^4 n}{n^2} \right).
 \]
 Thus, with probability at least $1-O\left(\dfrac{\log^4 n}{n}\right)$, any vertex is used as $w$
 at most three times. 
 
 Overall, we get that, with probability at least 
 \[ 
 		1 - e^{-\omega(\log n)} 
   -o\left(\dfrac{\log^2 n}{\sqrt{n}}\right) - O\left(\dfrac{\log^4 n}{n}\right) 
 		 \geq 1 -  o\left(\dfrac{\log^2 n}{\sqrt{n}}\right),
 \]
 the neighbourhoods of constructed graphs $\calG_{n,d}$ and  $\calG_{n,d}^h$ differ by at most $8=2+2\cdot 3$ elements.




\begin{thebibliography}{10}




\bibitem{AGSW} P. Ara\'{u}jo, S. Griffiths, M. \v{S}ileikis, L. Warnke, Extreme local statistics in random graphs: maximum tree extension counts,  e-preprint arXiv:2310.11661.


\bibitem{BES} L. Babai, P. Erd\H{o}s, S. M. Selkow, Random graph isomorphism, {\it SIAM Journal on Computing}, {\bf 9}(3) (1980) 628--635.

\bibitem{BH2013}
A. Barvinok, J.A. Hartigan, The number of graphs and a random graph with a given degree sequence, \textit{Random Structures Alg.}, \textbf{42} (2013) 301--348.

\bibitem{Bol1} B. Bollob\'{a}s, The distribution of the maximum degree of a random graph, {\it Discrete Mathematics} {\bf 32} (1980) 201--203.


\bibitem{Bol2} B. Bollob\'{a}s, Degree sequences of random graphs, {\it Discrete Mathematics} {\bf 33} (1981) 1--19.

\bibitem{CZL2012}
 A.L.~Chen, F.J.~Zhang, H.~Li,.
The degree distribution of the random multigraphs,		
\textit{Acta Mathematica Sinica, English Series},
\textbf{28}(5) (2012), 941--956.

\bibitem{Gao} P. Gao,  
 Triangles and subgraph probabilities in random regular graphs, \textit{Elec. J. Comb}, \textbf{31}(1) (2024), article number P1.2.
  

\bibitem{GW2017}
E.~Greene, J. A.~Wellner, 
Exponential bounds for the hypergeometric distribution,  
\textit{Bernoulli}, \textbf{23}(3) (2017), 1911--1950.

\bibitem{IM2018}
M.~Isaev and B.\,D.~McKay,
Complex martingales and asymptotic enumeration,
\emph{Random Structures \& Algorithms} {\bf 52} (2018) 617--661.

\bibitem{IMSZ}
M.~Isaev, B. D. McKay, A. Southwell, M. E. Zhukovskii,
Sprinkling with random regular graphs, 
e-preprint arXiv:2309.00190.


\bibitem{IRZZ}  M.~Isaev, I.~Rodionov, R.-R.~Zhang, M. E.~Zhukovskii, \textit{Extremal independence in discrete random systems,  Annales de l’Institut Henri Poincar\'e B: Probability and  Statistics}, \textbf{60}(4), (2024), 2923--2944.

\bibitem{Ivchenko1973}
G.I. Ivchenko, On the asymptotic behavior of degrees of vertices in a random graph, \textit{Theory of probability \& its applications}, \textbf{18}(1) (1973), 188--195.


\bibitem{KSVW2001}
M.~Krivelevich, B.~Sudakov, V.H.~Vu, N.C~Wormald,  Random regular graphs of high degree, 
\textit{Random Struct. Alg.} \textbf{18} (2001) 346--363.

\bibitem{nadarajah2002}
S. Nadarajah, K. Mitov.
\newblock Asymptotics of maxima of discrete random variables.
\newblock {\em Extremes}, {\bf 5}(3) (2002) 287--294.


\bibitem{McKay2011}
B. D. McKay, Subgraphs of dense random graphs with specified degrees, 
\textit{Combinatorics, Probability and Computing}, \textbf{20} (2011) 413--433.

\bibitem{MW-indep} B. D. McKay, N. C. Wormald, The degree sequence of a random graph. I. The models, {\it Random Struct. Alg.} {\bf 11}:2 (1997) 97--117.

\bibitem{Palka1987}
Z.~Palka, Extreme degrees in random graphs,
\textit{Journal of Graph Theory}, \textbf{11}(2) (1987),  
121--134.
 



\bibitem{RZ} I. V. Rodionov, M. E. Zhukovskii, The distribution of the maximum number of common neighbors in the random graph, {\it European Journal of Combinatorics} {\bf 107} (2023) 103602.

\bibitem{SS} A. Sah, M. Sawney, Subgraph distributions in dense random regular graphs, {\it Compositio Mathematica} {\bf 159} (10) (2023), 2125--2148.


\bibitem{Shang2012}
Y.~Shang, 
Focusing of maximum vertex degrees in random faulty scaled sector graphs,
\textit{Panamerican Mathematical Journal}, \textbf{22}(2) (2012), 1--17.



\end{thebibliography}
\end{document}